\documentclass[a4paper,11pt]{article}

\input{styleart.sty}
            
\date{}

\begin{document}

\author{Chlo\'e Perin and Rizos Sklinos}
\title{On the (non) superstable part of the free group}
\maketitle

\begin{abstract} 
In this short note we prove that a definable set $X$ over $\F_n$ is superstable only if $X(\F_n)=X(\F_{\omega})$.
\end{abstract}

\section{Introduction}
In this note we are concerned with the superstable part of the free group. Although Poizat showed in \cite{PoizatGenericAndRegular} that 
the theory of non abelian free groups is not superstable, 
it is still interesting to recover the superstable part of this theory, i.e. the 
definable sets that can be ranked under Shelah rank. 

If $\phi(\bar{x})$ is a first order formula over the free group $\F_n$ of rank $n$, we denote by $\phi(\F_n)$ the solution set of 
$\phi(\bar{x})$ in $\F_n$. We note that the following sequence 
$\F_2\prec\F_3\prec\ldots\prec\F_n\prec\ldots$ forms an elementary chain (see \cite{Sel6}, \cite{KharlampovichMyasnikov}), 
and consequently, for each $n\geq 2$, $\F_n$ is an elementary subgroup of $\F_{\omega}$. We conjecture:

\begin{conjIntro}
Let $\phi(\bar{x})$ be a first order formula over $\F_n$. Then $\phi(\bar{x})$ is superstable 
if and only if $\phi(\F_n)=\phi(\F_{\omega})$.
\end{conjIntro}

The main theorem of this short note is the left to right direction of the above conjecture.

\begin{thmIntro} \label{MainResult}
Let $\phi(\bar{x})$ be a first order formula over $\F_n$. If $\phi(\bar{x})$ is superstable 
then $\phi(\F_n)=\phi(\F_{\omega})$.
\end{thmIntro}

Note that Pillay proposed an alternative proof of Theorem \ref{MainResult}, which relies (indirectly) on the result 
about elimination of imaginaries proved by Sela in \cite{SelaIm}.

\section{Model theory of the free group} \label{Stability}

In this section, we recall some results about model theoretic properties of the free group. 
For the benefit of the general reader, we quickly recall the model theoretic definitions, 
but for more background on stable theories we refer the reader to \cite{PillayStability}. We denote by $\F_n$ the free group of rank 
$n$ with basis $e_1,\ldots,e_n$.

\subsection{Stability and generic types}

Recall that a theory $T$ is stable if no formula 
has the order property, i.e. there are no $(\bar{b}_n)_{n<\omega},(\bar{c}_n)_{n<\omega}$ and $\phi(\bar{x},\bar{y})$ 
such that $\mathcal{M}\models \phi(\bar{b}_i,\bar{c}_j)$ if and only if $i<j$, in any model $\mathcal{M}$ of $T$.

The following is a deep result of Sela \cite{SelaStability}. 

\begin{theorem}
Let $\F$ be a non-abelian free group. Then $\mathcal{T}h(\F)$ is stable.
\end{theorem}

For background on stable groups we suggest \cite{PoizatStableGroups}. We recall that a formula $\phi$ with one free variable in 
a stable group is {\em generic} if finitely many translates of the formula by elements of the group cover the whole group. 
A type of a stable group is called {\em generic} if it contains only generic formulas. Poizat proved that $\F_{\omega}$ is connected, i.e. 
it has no proper definable subgroup of finite index. Assuming stability his result is equivalent 
to saying that there exists a unique generic type (over any set of parameters). Moreover 
Pillay proved the following (see \cite{PillayGenericity}): 

\begin{theorem} \label{GenericType}
Let $\phi(x)$ be a formula over $\F_n$. Then $\phi$ is generic if and only if $\F_{n+1}\models \phi(e_{n+1})$.
\end{theorem}

\subsection{Equationality}

A stronger notion than stability is that of equationality. 
\begin{definition} Let $T$ be a complete first-order theory and 
let $\phi(\bar{x},\bar{y})$ be a first-order formula (without parameters). We say that $\phi$ is an {\em equation} 
(with respect to $\bar{y}$) if there exists $n<\omega$ such that the intersection of any family of instances of 
$\phi(\bar{x},\bar{y})$, i.e. $\{\phi(\bar{x},\bar{b}_i) \ |\ i\in I\}$, is equivalent (in $T$) to the intersection 
of a finite subfamily $\{\phi(\bar{x},\bar{b}_j) \ | \ j\in J\subseteq I\}$ with $\abs{J}\leq n$. 

A first order theory is called {\em equational} if any formula is equivalent to a boolean combination of instances 
of equations. 
\end{definition}
It is not hard to see that an equational first-order theory is stable. 

The first-order theory of a non-abelian free group is the first natural example of a stable theory 
which is not equational as proved by Sela in \cite{SelaEquational}. On the other hand, we have:

\begin{theorem}[\cite{SelaStability}]\label{Equa}
Let $\phi(\bar{x},\bar{y})$ be a Diophantine formula, i.e. 
$\phi(\bar{x},\bar{y}):=\exists \bar{z}(\Sigma(\bar{x},\bar{y},\bar{z})=1)$. 
Then $\phi$ is an equation.
\end{theorem}

\subsection{Superstable formulas}

We fix a stable first order theory $T$ and we work in a ``big'' 
saturated model $\mathbb{M}$. 

\begin{definition}
Let $\phi(\bar{x},\bar{b})$ be a first order formula over $\mathbb{M}$ and $A\subset\mathbb{M}$. Then 
$\phi(\bar{x},\bar{b})$ forks over $A$ if there is an infinite indiscernible sequence $(\bar{b}_i)_{i<\omega}$ with 
$tp(\bar{b}_1/A)=tp(\bar{b}/A)$, 
such that the set $\{\phi(\bar{x},\bar{b}_i): i<\omega\}$ is inconsistent. 
\end{definition}

The following remark is immediate.

\begin{remark}\label{Sub}
Let $\mathbb{M}\models \phi(\bar{x})\rightarrow \psi(\bar{x})$. 
Suppose $\psi(\bar{x})$ forks over $A$, then $\phi(\bar{x})$ forks over $A$.
\end{remark}

\begin{definition}
Let $\phi(\bar{x}),\psi(\bar{x})$ be two consistent first order formulas (possibly over parameters). We say $\phi<\psi$, if 
$\mathbb{M}\models \phi(\bar{x})\rightarrow\psi(\bar{x})$ and there is a set of parameters $A$, such that $\psi(\bar{x})$ is 
definable over $A$ and $\phi(\bar{x})$ forks over $A$.
\end{definition}

\begin{definition}
$\phi$ is superstable if there is no infinite descending chain of formulas $\phi:=\phi_1>\phi_2>\ldots>\phi_n>\ldots$ . 
\end{definition}

%

\section{Formal solutions}

For convenience, whenever $\bar{x}:=(x_1,x_2,\ldots,x_n)$ is a tuple of letters in some vocabulary, we denote by $\F(\bar{x})$ the 
free group freely generated by $x_1,\ldots,x_n$. Whenever we have a system of equations $\Sigma(\bar{x},\bar{y})=1$ over a group $P$, 
we will denote by $G_{\Sigma}$ the group $\langle \bar{x},\bar{y}, P \ | \ \Sigma(\bar{x},\bar{y})\rangle$.

\begin{definition} 
Let $\Sigma(\bar{x}, \bar{y})=1$ be a system of equations with parameters in a group $P$. 
A formal solution for $\Sigma$ over $P$ is a tuple $\tilde{y}(\bar{x}, \bar p)$ of $P * \F(\bar{x})$ 
such that $\Sigma(\bar{x}, \tilde{y}(\bar{x}, \bar p))=1$ holds in $P * \F(\bar{x})$. 
\end{definition}

Note that if a formal solution exists, then for any $\bar{x}_0 \in P$, the equation 
$\Sigma(\bar{x}_0, \bar{y})=1$ has a solution in $P$, given by $\tilde{y}(\bar{x}_0, \bar p)$. 

Also, this implies that the map $r:G_{\Sigma} \to P * \F(\bar{x})$ given by 
$\bar{x} \mapsto \bar{x}$, $P \stackrel{Id}{\mapsto} P$ and $\bar{y} \mapsto \tilde{y}$ is a morphism. 
By definition, it sends $\langle P, \bar{x} \rangle_{G_\Sigma}$ surjectively onto $P * \F(\bar{x})$, 
hence $\langle P, \bar{x} \rangle_{G_\Sigma}$ is exactly $P * \F(\bar{x})$ and $r: G_{\Sigma} \to P * \F(\bar{x})$ is a retraction. 

The following lemma tells us that if the Diophantine condition over $\F_n$, $\exists \bar{y} \Sigma(x, \bar{y}) = 1$,  defines 
a generic set, then $\Sigma$ admits a formal solution over $\F_n$.

\begin{lemma}\label{GenericMerzlyakov}
Let $\Sigma(x,\bar{y})=1$ be a system of equations over $\F_n$. 
Suppose that $\F_{n+1} \models \exists\bar{y}(\Sigma(e_{n+1},\bar{y})=1)$. 

Then there exists a formal solution for $\Sigma$ over $\F_n$.
In particular $\F_n\models \forall x\exists\bar{y}(\Sigma(x,\bar{y})=1)$.
\end{lemma}

\begin{proof} Let $\tilde{y}(e_1, \ldots, e_n, e_{n+1})$ be an element of $\F_{n+1}$ such that $\Sigma(e_{n+1}, \tilde{y}(e_1, \ldots, e_n, e_{n+1})) = 1$ in $\F_{n+1}$. Clearly $\Sigma(x, \tilde{y}(e_1, \ldots, e_n, x)) = 1$ in $\F_n * \F(x)$, so 
$\tilde{y}(e_1, \ldots, e_n, x)$ is a formal solution for $\Sigma$ over $\F_n$. 
\end{proof}

The following generalization is straightforward. 

\begin{lemma}\label{GenerictupleMerzlyakov}
Let $\Sigma(\bar{x},\bar{y})=1$ be a system of equations over $\F_n$ with $\abs{\bar{x}}=m$. Suppose 
$\F_{n+m}\models\exists\bar{y}(\Sigma(e_{n+1},\ldots, e_{n+m},\bar{y})=1)$. 
Then there exists a formal solution for $\Sigma$ over $\F_n$.

In particular $\F_n\models \forall \bar{x} \exists\bar{y}(\Sigma(\bar x,\bar{y})=1)$.
\end{lemma}

The above lemma has the following immediate corollary we believe well known.
\begin{corollary}
Let $\phi(\bar{x})$ be a positive formula over $\F_n$ such that $\F_{n+m}\models \phi(e_{n+1},\ldots, e_{n+m})$. 
Then $\phi(\F_n)$ is $\F_n^m$.
\end{corollary}
In other words, there is no proper positive formula which is generic.

\begin{proof}
It is enough to note that in \cite{Mer} it was proved that a positive formula over $\F_n$ is equivalent to a Diophantine one. 
\end{proof}

\section{Non-superstable sets in the free group}\label{OurProof}

We start this section by recalling from \cite{SklinosGenericType} the following recursively defined sequence of elements $b_i$ of $\F_{\omega}$:
$$b_1=e_1e_2^2$$ 
$$b_{i}=b_{i-1}\cdot e_{i}e_{i+1}^2$$ 
The elements $b_i$ form an independent set of realization of the generic type of the free group, 
indeed, the finite subset $\{b_1, \ldots, b_k\}$ together with $e_{k+1}$ form a basis of $\F_{k+1}$. In particular the above sequence is an indiscernible set.

In \cite{ThesisSklinos} it was proved that the first order theory of the free group is not Diophantinely-superstable, i.e. that
one can find an infinite descending chain $\psi_1>\psi_2>\ldots>\psi_m>\ldots$ of Diophantine formulas. 
Inspired by the proof of this,  we give an alternative proof of the following fact.
\begin{proposition}\label{Generic}
Let $\phi(x)$ be a generic formula over a free group $\F(a_1, \ldots, a_n)$. Then $\phi(x)$ is not superstable.
\end{proposition}

\begin{proof}
We consider the sequence $(b_k)_{k<\omega} $ in $\F_{\omega} = \F(e_1, e_2, \ldots)$. We also consider the following sequence of formulas over $\F_{\omega}$:

\begin{tabular}{l}
$\psi_1(x,b_1):=\exists y(x=b_1y^{-2})$\\
$\psi_2(x,b_1,b_2):=\exists y(x=b_1(b_1^{-1}b_2y^{-2})^{-2})$\\
$\vdots$\\ 
$\psi_{i}(x,b_1,\ldots,b_i):= \exists y(x=b_1(b_1^{-1}b_2(b_2^{-1}b_3(\ldots(b_{i-1}^{-1}b_i y^{-2})^{-2})\ldots)^{-2})^{-2})$\\
$\vdots$\\
\end{tabular}

Note that $\F_{\omega} \models \psi_i(e_i)$ so these formulas are consistent. It is easy to see that $\models\psi_{i+1}\rightarrow\psi_i$. 
Moreover, it was proved in \cite{ThesisSklinos} that $\psi_{i}$ forks over $b_1,\ldots,b_{i-1}$. 

Let $\phi_i:=\psi_i \land \phi$: since $e_i$ is generic over $\F(\bar{a})$, 
the formula $\phi_i$ is satisfied by $e_{i}$ in $\F(\bar a)*\F_{\omega}$, hence it is consistent. By Remark \ref{Sub} we have that $\phi_i$ forks over $\F(\bar{a}),b_1,\ldots,b_{i-1}$, so we have that $\phi>\phi_1>\phi_2>\ldots>\phi_m>\ldots$ as we wanted.
\end{proof}

We generalize this to show
\begin{theorem}
Let $\phi(\bar{x})$ be a first order formula over $\F:=\F(a_1,\ldots, a_n)$. Suppose $\phi(\F)\neq \phi(\F*\F_{\omega})$. Then $\phi(\bar{x})$ is not superstable.
\end{theorem}

\begin{proof} 
Throughout the proof the tuple $\bar{a}$ is always assumed to be part of the parameter set of the various formulas, whether we explicitly mention it or not. 
Without loss of generality $\abs{\bar{x}}=1$. 
We may also assume that $\phi(\F)\neq \phi(\F * \langle e_1 \rangle)$ otherwise we can trivially think of 
$\phi(x)$ as being definable over a bigger subfactor of $\F*\F_{\omega}$. 

Suppose $w(\bar{a},e_{1})\in \phi(\F * \langle e_1 \rangle)\setminus\F$. 
For each $i<\omega$, we consider the following (Diophantine) formula 
$$\phi_i(x, b_1,\ldots , b_i):= \exists z(x=w(\bar{a},z) \land \psi_i(z))$$ 
where $\psi_i$ is the formula defined in the proof of Proposition \ref{Generic}. It is enough to show that $\phi_1>\phi_2>\ldots>\phi_m>\ldots$

We define the word
$\gamma_i(y, b_i) = b_1(b_1^{-1}b_2(b_2^{-1}b_3(\ldots(b_{i-1}^{-1}b_i y^{-2})^{-2})\ldots)^{-2})^{-2})$ 
(it depends on $b_1, \ldots, b_{i-1}$ as well but wo do not make this explicit), so that $\phi_i(x, b_1,\ldots , b_i)$ can be rewritten as
$$\phi_i(x, b_1,\ldots , b_i):=  \; \exists y  \; x= w(\bar{a}, \gamma_i(y, b_i))$$

We fix some $i<\omega$ and we prove that $\phi_{i+1}$ forks over $\bar{a},b_1,\ldots,b_i$. We claim that the (indiscernible) 
sequence $(b_{i+m})_{m\geq1}$ is a witnessing sequence for forking. Suppose not, then $\{\phi_{i+1}(x,b_1,\ldots, b_i, b_{i+m}) \ |\ m \geq 1\}$ is consistent. 
By Theorem \ref{Equa}, $\phi_{i+1}(x,\bar{y})$ is equational so there is some $k<\omega$ such that: 

$$\bigwedge \{\phi_{i+1}(x, b_1,\ldots,b_i,b_{i+m})| m\geq 1\}
= \bigwedge^k_{j=1} \phi_{i+1}(x, b_1\ldots,b_i, b_{i+j})$$

Denote by $\bar{b}$ the tuple $(b_1,\ldots,b_{i+k})$. The finite intersection above is non empty, 
so it contains an element $c$ which we can assume to be in $\F*\F(\bar{b})$. Thus for any $m\geq 1$, there exists $y$ such that $c = w(\bar{a}, \gamma_i(y, b_{i+m}))$.
We consider the word $\Sigma(y,x, \bar{b}) = c^{-1} w(\bar{a}, \gamma_i(y, x))$.

The Diophantine set $\exists y(\Sigma(y,x,\bar{b})=1)$ (over $\F*\F(\bar b)$) is generic since it contains 
for example the element $b_{i+k+1}$ which is generic over $\bar{a}, \bar{b}$. 
Thus, applying Theorem \ref{GenericMerzlyakov} we get a formal solution 
$\tilde{y}(x, \bar{b})$ for $\Sigma$ over $\F * \F(\bar{b})$. 

\ \\ \\
{\it Claim I:} The element $\gamma_i(\tilde{y}(x, \bar{b}), x)$ of $\F * \F(\bar{b})* \F(x)$ does not belong to $\F * \F(\bar{b})$.
\begin{proof}[Proof of Claim I]
Consider the homomorphism $h: \F * \F(\bar{b}) * \F(x) \rightarrow \F(x)$ 
killing $\F * \F(\bar{b})$ and sending $x$ to $x$. We have $h(\tilde{y}) = x^p$ for some $p$, so: 
$$h(\gamma_i(\tilde{y}, x)) = (((\ldots(x x^{-2p})^{-2})\ldots)^{-2})^{-2}$$   
so that $h(\gamma_i(\tilde{y},x)) \neq 1$: in particular, $\gamma_i(\tilde{y},x)$ is not in $\F * \F(\bar{b})$.
\end{proof}
\ \\
{\it Claim II:} The subgroup $H$ of $\F * \F(\bar{b})* \F(x)$ generated by $\F * \F(\bar{b})$ together with $\gamma_i(x,\tilde{y})$ is the free product
$$ \F * \F(\bar{b}) *  \langle \gamma_i(x, \tilde{y}) \rangle$$ 
\begin{proof}[Proof of Claim II]
Note that $H$ is a free group, generated by $\abs{\bar{a}} + \abs{\bar{b}} +1$ elements, and by Kurosh' theorem 
$\F * \F(\bar{b})$ is a free factor of $H$ which by Claim I is proper. 
Thus $H$ is a free group of rank strictly greater than $\abs{\bar{a}} + \abs{\bar{b}} $, which proves the claim. 
\end{proof}

To conclude, by Claim II we see that $c^{-1}w(\bar{a},\gamma_i(x, \tilde{y})) \neq 1$ in $H$, 
since $c^{-1} \in \F * \F(\bar{b})$ and $w(\bar{a},\gamma_i(\tilde{y}, x))\in \F * \langle \gamma_i(\tilde{y}, x) \rangle\setminus \F$ by hypothesis on $w$. 
Thus in particular this is also true in the bigger group
$\F(x) *\F * \F(\bar{b})$. This contradicts the fact that $\tilde{y}$ is a formal solution for this equation. 

\end{proof}

\bibliography{biblio}

\end{document}